\documentclass[11pt,reqno,a4paper]{amsart}
\usepackage{amsfonts,amsmath,times,amsthm,amssymb}
\usepackage{tikz}
\usepackage{cool} 

\newtheorem{theorem}{Theorem}
\newtheorem{lemma}{Lemma}

\newtheorem{prop}{Proposition}
\newtheorem{corollary}{Corollary}

\renewcommand{\E}{\mathbb{E}}
\newcommand{\Var}{\mathbb{V}{\rm ar}}
\newcommand{\Prob}{\mathbb{P}}

\newcommand{\field}{\mathbb{F}}
\newcommand{\indicator}{\boldsymbol{1}}
\newcommand{\given}{\, \vert \,}
\newcommand{\convL}{\, \overset{L_1}{\longrightarrow} \,}
\newcommand{\convP}{\, \overset{P}{\longrightarrow} \,}
\newcommand{\convLL}{\, \overset{L_2}{\longrightarrow} \,}
\newcommand{\convD}{\, \overset{D}{\longrightarrow} \,}

\newcommand{\degree}{{\rm deg}}

\begin{document}
\begin{center}
	{\Large \bf  
	The Zagreb index of several random models}
	
	\bigskip
	{\bf Panpan Zhang}
	
	\bigskip
	{\tiny Department of Biostatistics, Epidemiology and Informatics, Perelman School of Medicine, University of Pennsylvania, Philadelphia, PA 19104, U.S.A.}
	
	\bigskip
	
	\today
\end{center}

\bigskip\noindent
{\bf Abstract.}
In this article, we investigate the Zagreb index, a kind of graph-based topological index, of several random networks, including a class of networks extended from random recursive trees, plain-oriented recursive trees, and random caterpillars growing in a preferential attachment manner. We calculate the mean and variance of the Zagreb index for each class. In addition, we prove that the asymptotic distribution of the Zagreb index for the first class is normal, and that the asymptotic distribution of the Zagreb index for the second class is skewed to the right. 


Secondary: 60B10; 60F05


\section{Introduction}
\label{Sec:Intro}

A topological index, in chemical graph theory, is a metric that quantifies the structure of the molecular graph of a chemical compound via a number. The {\em Zagreb index}~\cite{Gutman} is a topological index that has found a plethora of applications in mathematical chemistry and chemoinformatics. It is best known for modeling quantitative structure-property relationship (QSPR) and quantitative structure-activity relationship (QSAR) between molecules~\cite{Todeschini}. The Zagreb index of a graph, $G = (V, E)$, is the sum of the squared degrees of all the nodes in $G$. Mathematically, it is given by
$${\tt Zagreb}(G) = \sum_{v \in V} \bigl(\degree(v)\bigr)^2,$$
where $\degree(v)$ is the degree of node $v$.

Recently, the Zagreb index of several random trees were investigated, such as random recursive trees (RRTs)~\cite{Feng2011} and $b$-ary search trees~\cite{Feng2015}. In this article, we calculate the Zagreb index of three random structures; They are a class of networks extended from RRTs, plain-oriented recursive trees (PORTs), and a class of caterpillars growing in a {\em preferential attachment} manner.

\section{Zagreb index of extended RRTs}

Tree is a popular structure for data storage and sorting in computer science. A rooted tree is a tree in which there is one designated node called {\em root}. The root of a tree is usually thought of as the originator of the tree. A {\em random recursive tree} (RRT) is a non-planer rooted tree such that a node is uniformly chosen from all the nodes in the existing tree as a parent for a new child at each growth step. The children of any parent in a RRT are not ordered. 

The Zagreb index of RRTs was investigated by~\cite{Feng2011}. The exact mean and variance of the index were calculated. They both increase linearly with respect to time $n$. The asymptotic distribution of the Zagreb index (scaled by $n$) was proven to follow a Gaussian law. In this section, we look into the Zagreb index of a class of networks extended from RRTs. This class of networks evolve as follows. At time $n = 1$, there is a total of $m_0 \ge 1$ nodes that are mutually connected by edges. If there is a single node ($m_0 = 1$) at the initial point, it exists as an isolated node, where no self loop is considered. At each subsequent time point, we randomly choose $m \le m_0$ (distinct) nodes from the existing network and connect them with a newcomer by $m$ edges. Our goal is to study the Zagreb index of this class of networks at time $n$, denoted by $U_n$. A RRT appears as a special case of this network by setting $m_0 = m = 1$.

We enumerate all the nodes in $U_n$ in the following way. We label the initial $m_0$ nodes with distinct numbers in $\{1, 2, \ldots, m_0\}$. Before recruiting any child, these $m_0$ nodes are structurally equivalent, so the order of labeling is arbitrary. The child that joins the network at time $n \ge 2$ is labeled with $(m_0 + n - 1)$. Thus, there is a total of $(m_0 + n - 1)$ nodes in $U_n$. For each $j = 1, 2, \ldots, m_0 + n - 1$, let $D_{n, j}$ be the degree of the node labeled with $j$. In addition, let 
$$Z_n = {\tt Zagreb}(U_n) = \sum_{j = 1}^{m_0 + n - 1} D^2_{n, j}$$ 
be the Zagreb index of $U_n$. Note that we will repeatedly use $D$ and $Z$ with proper subscripts as node degrees and the Zagreb index for all kinds of random graphs investigated through this manuscript. In the next proposition, we calculate the expectation of $Z_n$, and develop a weak law as well.

\begin{prop}
	\label{Prop:EZU}
	For $n \ge 1$, the mean of the Zagreb index of $U_n$ is
	$$ \E[Z_n] = \frac{(5m^2 + m)n^2 - 2mm_0(2m - m_0 + 1)n\log{n} + O(n)}{n + m_0 - 1}.$$
	As $n \to \infty$, we have
	$$\frac{Z_n}{n} \convL 5m^2 + m.$$
	This convergence takes place in probability as well.
\end{prop}

\begin{proof}
	Let $\field_{n - 1}$ denote the $\sigma$-filed generated by the history of the first $(n - 1)$ stages of the network, and let $\indicator_{n, S}$ denote the event indicating that nodes labeled with the indices in set $S$ are chosen as parents for the new child at time $n$. Upon the insertion of node $n$, an almost-sure relation between $Z_n$ and $Z_{n - 1}$, conditional on $\field_{n - 1}$ and $\indicator_{n, S}$ is given by
	$$Z_{n} = Z_{n - 1} + \sum_{j = s_1}^{s_m} (D_{n, j} + 1)^2 - \sum_{j = s_1}^{s_m} D^2_{n, j} + m^2,$$
	where $S := \{s_1, s_2, \ldots, s_m\}$ is an $m$-long subset of $\{1, 2, \ldots, m_0 + n - 2\}$. We simplify the almost-sure relation to get
	\begin{equation}
	\label{Eq: recZam}
	Z_n = Z_{n - 1} + 2 \sum_{j \in S} D_{n - 1, j} + (m^2 + m).
	\end{equation}
	Taking the average over all possible $S$'s, we obtain
	\begingroup
	\allowdisplaybreaks
	\begin{align*}
	\E[Z_n \given \field_{n - 1}] &= Z_{n - 1} + \frac{2}{{n + m_0 - 2 \choose m}} \sum_{S} \sum_{j \in S} D_{n - 1, j} + (m^2 + m)
	\\ &= Z_{n - 1} + 2 \times \frac{{n + m_0 - 3 \choose m - 1}}{{n + m_0 - 2 \choose m}} \sum_{j = 1}^{n + m_0 - 2} D_{n - 1, j} + (m^2 + m)
	\\ &= Z_{n - 1} + \frac{2m}{n + m_0 - 2} \sum_{j = 1}^{n + m_0 - 2} D_{n - 1, j} + (m^2 + m),
	\end{align*}
	\endgroup
	where the sum $\sum_{j = 1}^{n + m_0 - 2} D_{n - 1, j}$ is not random; It is equal to $\bigl(m_0(m_0 - 1) + 2m(n - 1)\bigr)$. We thus can take another expectation with respect to $\field_{n - 1}$ to get a recurrence for $\E[Z_n]$, which is given by
	\begin{equation}
	\label{Eq:recU}
	\E[Z_n] = \E[Z_{n - 1}] + \frac{m \bigl( (5m + 1)n + 2m_0^2 + (m - 1)m_0 - 2(5m + 1)\bigr)}{n + m_0 - 2}.
	\end{equation}
	We solve this recurrence with the initial condition $\E[Z_1] = Z_1 = m_0(m_0 - 1)^2$, and get the result stated in the proposition.
	
	In what follows, we have
	$$|(n + m_0 - 1) Z_n - (5m^2 + m)n^2| = O_{L_1}(n \log{n}).$$
	Divide by $n^2$ on both sides, and let $n$ go to infinity. We obtain an $L_1$ converge for $Z_n / n$, as well as an in-probability convergence required for the weak law.
\end{proof}

The computation of the second moment of $Z_n$ is based on squaring the almost-sure relation of $Z_n$ presented in Equation~(\ref{Eq: recZam}). That is
\begin{align}
Z_n^2 &= Z_{n - 1}^2 + 4\left(\sum_{j \in S} D_{n - 1, j}\right)^2 + (m^2 + m)^2 +  4 Z_{n - 1} \sum_{j \in S} D_{n - 1, j} \nonumber
\\ &\qquad{}+ 2(m^2 + m)Z_{n - 1} + 4(m^2 + m) \sum_{j \in S} D_{n - 1, j}. \label{Eq:recZsq}
\end{align}

As done in the proof of Proposition~\ref{Prop:EZU}, we tend to take the expectation with respect to $\indicator_{n, S}$, then to take another expectation with respect to $\field_{n - 1}$, and ultimately to get a recurrence for the second moment of $Z_n$. Before implementing this strategy, we take the most complex term in Equation~(\ref{Eq:recZsq}) out and simplify it separately as follows:
$$
\sum_{S}\left(\sum_{j \in S} D_{n - 1, j}\right)^2 = \sum_{S}\left(\sum_{j \in S} D^2_{n - 1, j} + 2 \sum_{j \neq k \in S}D_{n - 1, j}D_{n - 1, k}\right).$$
The first part is simple. It is
$$ \sum_{S} \sum_{j \in S} D^2_{n - 1, j} = {n + m_0 - 3 \choose m - 1} Z_{n - 1}.$$
The second part is
\begingroup
\allowdisplaybreaks
\begin{align*}
&2\sum_S\sum_{j \neq k \in S} D_{n - 1, j} D_{n - 1, k} 
\\&\quad= 2 {n + m_0 - 4 \choose m - 2}\sum_{j \neq k} D_{n - 1, j} D_{n - 1, k}
\\&\quad= {n + m_0 - 4 \choose m - 2} \sum_{j = 1}^{n + m_0 - 2} D_{n - 1, j} \left(\sum_{j = 1}^{n + m_0 - 2} D_{n - 1, j}  - D_{n - 1, j}\right)
\\&\quad= {n + m_0 - 4 \choose m - 2} \sum_{j = 1}^{n + m_0 - 2} D_{n - 1, j} \bigl(m_0(m_0 - 1) + 2m(n - 2) - D_{n - 1, j}\bigr)
\\ &\quad =  {n + m_0 - 4 \choose m - 2} \left(\bigl(m_0(m_0 - 1) + 2m(n - 2)\bigr)^2 - Z_{n - 1} \right)
\end{align*}
\endgroup

We are now ready to derive the second moment of $Z_n$, the result of which is presented in the next proposition.

\begin{prop}
	\label{Prop:EZsqU}
	For $n \ge 1$, the second moment of the Zagreb index of $U_n$ is
	$$\E\left[Z^2_n\right] = \frac{(5m^2 + m)^2 n^5 - 4m^2m_0(5m + 1)(2m - m_0 + 1)n^4\log{n} + O\left(n^4\right)}{(n + m_0 - 2)(n + m_0 - 1)^2}.$$
\end{prop}

\begin{proof}
	Recall the squared almost-sure relation in Equation~(\ref{Eq:recZsq}). Take the expectation with respect to $\indicator_{n, S}$ to get
	$$
	\E\left[Z_n^2 \given \field_{n - 1}\right] = Z_{n - 1}^2 + C_1(n, m, m_0)Z_{n - 1} + C_2(n, m, m_0),
	$$
	where $C_1$ and $C_2$ are two constant functions (free of $Z_{n}$) depending only on $n$, $m$ and $m_0$. We have the exact expressions of $C_1$ and $C_2$, but they are too lengthy to report in the manuscript. Taking another expectation with respect to $\field_{n - 1}$ and plugging in $\E[Z_n]$ derived in Proposition~\ref{Prop:EZU}, we obtain a recurrence for the second moment of $Z_n$. Solving the recurrence with initial condition $\E\left[Z^2_1\right] = Z_1^2 = m_0^2(m_0 - 1)^4$, we get the result stated in the proposition. 
\end{proof}

Although we only present the first two leading terms of $\E\left[Z_n^2\right]$ in Proposition~\ref{Prop:EZsqU}, we obtain the exact expressions of a few more terms in the calculation. These terms are needed to determine the order of the leading term of the variance of $Z_n$. These terms are available upon request by the readers. In the next corollary, we give the variance of $Z_n$, computed by taking the difference between the second moment and the squared first moment of $Z_n$.

\begin{corollary}
	\label{Cor:VarZ}
	For $n \ge 1$, the variance of the Zagreb index of $U_n$ is
	$$
	\Var[Z_n] = 4m^2(m + 1)n + 4m^2m_0(m_0 - 2m - 1) \log^2{n} + O\left(\log{n}\right).
	$$
\end{corollary}

In~\cite{Feng2011}, the authors proved that the variance of the Zagreb index of RRT is asymptotically equal to $8n$, which is the special case of Corollary~\ref{Cor:VarZ} ($m = 1$). According to Corollary~\ref{Cor:VarZ}, we find that the variance of the Zagreb index of $U_n$ is linear in $n$, and its asymptotic value does not depend on $m_0$. In the next corollary, we show that $Z_n/n$ converges to $5m^2 + m$ in $L_2$-space (stronger than the $L_1$ convergence and the in-probability convergence presented in Proposition~\ref{Prop:EZU}).

\begin{corollary}
	As $n \to \infty$, we have
	$$\frac{Z_n}{n} \convLL 5m^2 + m.$$
\end{corollary}

\begin{proof}
	According to the asymptotic mean and variance of $Z_n$, we have
	\begin{align*}
	\E\left[\bigl|Z_n - (5m^2 + m)n\bigr|^2\right] &= \E\left[\bigl(Z_n - \E[Z_n] + \E[Z_n] - (5m^2 + m)n\bigr)^2\right]
	\\ &\le \Var[Z_n] + O\left(\log^2{n}\right)
	\\ &= O(n),
	\end{align*}
	which completes the proof.
\end{proof}

As both the mean and the variance of $Z_n$ are linear in $n$, we suspect that the limiting distribution of $Z_n$ scaled by $n$ is normal for general $U_n$, not just for the class of RRTs~\cite{Feng2011}. To prove the conjecture, our strategy is to apply a {\em Martingale Central Limit Theorem} (MCLT). According to Equation~(\ref{Eq:recU}), $\{Z_n\}_n$ is not a martingale. We consider the following transformation such that the transformed array is a martingale.

\begin{lemma}
	\label{Lem:martingale}
	For $n \ge 1$, the sequence 
	$$M_n = Z_n - \frac{(5m^2 + m)n^2 + O(n \log{n})}{n + m_0 - 1}$$
	is a martingale.
\end{lemma}

\begin{proof}
	Given a sequence $\{\beta_n\}_n$, consider $M_n := Z_n + \beta_n$ such that $\{M_n\}_n$ is a martingale. We retrieve $\beta_n$ based off the fundamental martingale property, i.e.,
	\begin{align*}
	\E[Z_n + \beta_n \given \field_{n - 1}] &= Z_{n - 1} + \frac{(5m + 1)n + 2m_0^2 + (m - 1)m_0 - 2(5m + 1)}{n + m_0 - 2} 
	\\ &\qquad{}+ \beta_n
	\\ &= Z_{n - 1} + \beta_{n - 1}
	\end{align*}
	We thus obtain a recurrence for $\beta_n$. We solve the recurrence with an arbitrary choice of the initial value of $\beta_n$, e.g., $\beta_1 = 0$, to get the result stated in the lemma.
\end{proof}

There are different forms of MCLTs listed in~\cite{Hall}, based off different sets of conditions. We choose a MCLT that requires a {\em conditional Lindeberg's condition} and a {\em conditional variance condition} for our proof.

\begin{lemma}
	\label{Lem：Lindeberg}
	The conditional Lindeberg's condition is given by
	$$U_n := \sum_{j = 1}^{n} \E\left[\left(\frac{\nabla M_j}{\sqrt{n}}\right)^2 \indicator_{\{|\nabla M_j / \sqrt{n}| > \varepsilon\}} \, \Bigg{|} \, \field_{j - 1}\right] \convP 0.$$
\end{lemma}

\begin{proof}
	By the construction of the martingale, we have
	\begin{align*}
	|\nabla M_j| &= \bigl|Z_j + \beta_j - (Z_{j - 1} + \beta_{j - 1})\bigr|
	\\ &\le |Z_j - Z_{j - 1}| + \frac{(5m + 1)j + O(1)}{j + m_0 - 2}
	\\ &\le 2 \max_S \sum_{i \in S} D_{j - 1, i} + m^2 + m + \frac{(5m + 1)j + O(1)}{j + m_0 - 2}
	\\ &= O(\log{n}).
	\end{align*}
	The bound for the maximum degree of a node is obtained by an analog to the strong law developed in~\cite{Devroye}. Therefore, $|\nabla M_j / \sqrt{n}|$ is uniformly bounded for all $j$. In other words, for any $\varepsilon > 0$, there exists $n_0(\varepsilon) > 0$ such that the sets $\{|\nabla M_j / \sqrt{n}| > \varepsilon\}$ are empty for all $n > n_0(\varepsilon)$. In what follows, we conclude that $U_n$ converges to $0$ almost surely, which is stronger than the in-probability convergence required for the condition.
\end{proof}

\begin{lemma}
	\label{Lem:variance}
	The conditional variance condition is given by
	$$V_n := \sum_{j = 1}^{n} \E\left[\left(\frac{\nabla M_j}{\sqrt{n}}\right)^2 \, \Bigg{|} \, \field_{j - 1}\right] \convP \eta^2,$$
	where $\eta^2$ is a random variable that is either finite or converges almost surely. Particularly for our case, $\eta^2$ is equal to $4m^3 + 4m^2$.
\end{lemma}

\begin{proof}
	We rewrite $V_n$ as follows:
	\begin{align*}
	V_n &= \frac{1}{n} \sum_{j = 1}^{n} \E\left[\bigl(Z_j + \beta_j - (Z_{j - 1} + \beta_{j - 1})\bigr)^2 \, \Big| \, \field_{j - 1}\right]
	\\ &= \frac{1}{n} \sum_{j = 1}^{n} \Big( \E\left[(Z_j - Z_{j - 1})^2 \, \big| \, \field_{j - 1}\right] + 2 \, \E\bigl[(Z_j - Z_{j - 1})(\beta_j - \beta_{j - 1}) \, \big| \, \field_{j - 1}\bigr]
	\\ &\qquad{}+ \E\left[(\beta_j - \beta_{j - 1})^2 \, \big| \, \field_{j - 1}\right] \Big).
	\end{align*}
	We calculate the three expectations in the summand one after another. The first part is
	\begin{align*}
	\E\left[(Z_j - Z_{j - 1})^2 \, \big| \, \field_{j - 1}\right] &= \E\left[Z_j^2 \given \field_{j - 1}\right] + Z^2_{j - 1} - 2Z_{j - 1}\E[Z_j \given \field_{j - 1}] 	
	\\ &= Z_{j - 1}^2 + C_1(j, m, m_0)Z_{j - 1} + C_2(j, m, m_0) + Z_{j - 1}^2
	\\ &\qquad{}- 2Z_{j - 1} \bigl(Z_{j - 1} + C_3(j, m, m_0)\bigr),
	\end{align*}
	where $C_3(j, m, m_0) = m\bigl( (5m + 1)j + 2m_0^2 + (m - 1)m_0 - 2(5m + 1)\bigr)/(j + m_0 - 2)$. The second part is
	\begin{align*}
	2 \, \E\bigl[(Z_j - Z_{j - 1})(\beta_j - \beta_{j - 1}) \, \big| \, \field_{j - 1}\bigr] &= (\beta_j - \beta_{j - 1})\bigl(\E[Z_j \given \field_{j - 1}] - Z_{j - 1}\bigr)
	\\&= -2 \, C_3(j, m, m_0)^2.
	\end{align*}
	The third part is
	$$\E\left[(\beta_j - \beta_{j - 1})^2 \, \big| \, \field_{j - 1}\right] = C_3(j, m, m_0)^2.$$
	Plugging in the asymptotic values of $\E\left[Z^2_{j - 1}\right]$ and $\E[Z_{j - 1}]$, we get
	\begin{align*}
	V_n &= \frac{1}{n} \sum_{j = 1}^{n} \bigl((25m^4 + 14m^3 + 5m^2) - 2m^2(5m + 1)^2 + m^2(5m + 1)^2 
	\\&\qquad{}+ O\left(\log{j}/j\right)\bigr)
	\\ &= 4m^3 + 4m^2 + O_{L_1}\left(\log^2{n}/n\right)
	\\ &\convL 4m^3 + 4m^2,
 	\end{align*}
 	which is stronger than the in-probability convergence required for the conditional variance condition.
\end{proof}

\begin{theorem}
	As $n \to \infty$, we have
	$$\frac{Z_n - (5m^2 + m)n}{2m\sqrt{m + 1}\sqrt{n}} 
	\convD N(0, 1).$$
\end{theorem}
\begin{proof}
	Upon the verifications of the two conditions in Lemmata~\ref{Lem：Lindeberg} and~\ref{Lem:variance}, we have
	$$\frac{Z_n + \beta_n - (Z_1 + \beta_1)}{\sqrt{n}} \sim \frac{Z_n - (5m^2 + m)n}{\sqrt{n}} \convD \mathcal{N}(0, 4m^3 + 4m^2),$$
	by the MCLT. This is equivalent to the stated result in the theorem.
\end{proof}

\section{Zagreb index of PORTs}
\label{Sec:PORTs}

In contrast to a RRT, a {\em plain-oriented recursive tree} (PORT) accounts for orders in the growth process. One simple way to interpret its evolution is that the probability that a node is chosen as a parent for a new child is proportional to its {\em degree} in the current tree. Mathematically, it is given by
$$\Prob(\indicator_v) = \frac{\degree(v)}{\sum_{u \in V} \degree(u)},$$
where $\indicator_v$ indicates the event that node $v$ is chosen as a parent for the newcomer, and $V$ is the set of all nodes in the current tree. Therefore, PORTs are a class of nonuniform trees. As its evolutionary process coincides with an attractive network characteristic---preferential attachment~\cite{Barabasi}, PORTs are of substantial interest in the community.

The Zagreb index of PORTs was investigated in a recent article~\cite{ZhangArXiv}, where the exact mean and variance were determined. The authors claimed that the Zagreb index of PORTs does not follow a Gaussian law as time goes to infinity by showing a numeric experiment. In this paper, we provide a more rigorous proof in support of that conjecture.

Let $T_n$ be a PORT at time $n$. As one node joins the tree at each step, there is a total of $n$ nodes in the $T_n$. We label these $n$ nodes with $\{1, 2, \ldots, n\}$ according to the time point of their appearance in the tree. Let $D_{n, j}$ be the degree of node $j$ at time $n$. The Zagreb index of $T_n$ is given by
$$Z_n = {\rm Zagreb}(T_n) = \sum_{j = 1}^{n} D^2_{n, j},$$
where $D_{n, j}$, again, is the degree of the node labeled with $j$ in $T_n$. 

\begin{prop}[\cite{ZhangArXiv}]
	For $n \ge 1$, we have
	\begin{align*}
	&\E[Z_n] = 2(n - 1) \bigl(\Psi(n) + \gamma \bigr)
	\\ &\E[Y_n] = \frac{32 \Gamma(n + 1/2)}{\sqrt{\pi}\Gamma(n - 1)} - 6(n - 1)\left(\Psi(n) + \gamma + \frac{8}{3}\right),
	\\ &\E \left[Z^2_n\right] = 4(n \log{n})^2 + 8 \gamma \left(n^2 \log{n}\right) + \left(16 + 4 \gamma^2 - \frac{2 \pi^2}{3}\right) n^2 + O\left(n^{3/2}\right),
	\end{align*}
	where $\Psi(\cdot)$ is the digamma function, $\gamma$ is the Euler's constant, and $Y_n := \sum_{j = 1}^{n} D^3_{n, j}$ is a topological index summing the cubic degrees of all the nodes.
\end{prop}

Based off the simulation result in~\cite{ZhangArXiv}, we suspect that the asymptotic distribution of $Z_n$ is skewed to the right, violating the property of symmetry of normal. Therefore, it suffices to show that the skewness of $Z_n$ is not zero; in fact, it is always negative.

In probability theory, the {\em skewness} of a random variable is defined as its standardized third central moment, i.e,
$${\mathcal{S}}(Z_n) = \E \left[\left(\frac{Z_n - \mu_{Z_n}}{\sigma_{Z_n}}\right)^3\right] = \frac{\E\left[Z_n^3\right] - 3 \E[Z_n] \E\left[Z_n^2\right] + 2 \bigl(\E[Z_n]\bigr)^3}{\bigl(\Var[Z_n]\bigr)^{3/2}},$$
in which most of the elements have already been determined, except for the third moment of $Z_n$. We resort to a recurrence method to calculate $\E\left[Z_n^3\right]$ exactly. To construct a recurrence for $\E\left[Z_n^3\right]$, we need the results presented in the next two lemmata. The first lemma is based on a new topological index, $X_n$, the sum of the degrees to the fourth power of all the nodes in $T_n$, i.e., $X_n := \sum_{j = 1}^{n} D^4_{n, j}$. 

\begin{lemma}
	\label{Lem:X}
	For $n \ge 2$, we have
	$$\E[X_n] = 36(n - 1)\left(n +\frac{7 \Psi(n)}{18} + \frac{7 \gamma}{18} + \frac{5}{3}\right) - \frac{192 \Gamma(n +1/2)}{\sqrt{\pi} \Gamma(n - 1)}.$$
	As $n \to \infty$, we have
	$$\frac{X_n}{n^2} \convL 36.$$
	The convergence takes place in probability as well.
\end{lemma}

\begin{proof}
	Upon time $n$ (node $n$ is not yet inserted), there is a total of $(n - 1)$ nodes in the current tree. In addition, the total of node degrees is $2(n - 2)$. By the definition of $X_n$, we have the following almost-sure relations from $X_{n - 1}$ (right before the insertion of node $n$) to $X_n$ (right after the insertion of node $n$), conditional on $\field_{n - 1}$ and $\indicator_{n, j}$, the event indicating that node $j$ is chosen as a parent for node $n$:
	\begin{align}
	X_n &= X_{n - 1} + (D_{n - 1, j} + 1)^4 - D^4_{n - 1, j} + 1 \nonumber
	\\ &= X_{n - 1} + 4D^3_{n - 1, j} + 6D^2_{n - 1, j} +4D_{n - 1, j} + 2.	\label{Eq:recX}
	\end{align}
	We average Equation~(\ref{Eq:recX}) out over $j$ to get
	\begin{align*}
	\E[X_n \given \field_{n - 1}] &= X_{n - 1} + 4\sum_{j = 1}^{n - 1}\frac{D^4_{n - 1}}{2(n - 2)} + 6\sum_{j = 1}^{n - 1}\frac{D^3_{n - 1}}{2(n - 2)} 
	\\ &\qquad{}+ 4\sum_{j = 1}^{n - 1}\frac{D^2_{n - 1}}{2(n - 2)} + 2
	\\ &= \frac{n}{n - 2} X_{n - 1} + \frac{3}{n - 2} Y_{n - 1} + \frac{2}{n - 2} Z_n + 2.
	\end{align*}
	Taking another expectation and plugging in the results of $\E[Y_n]$ and $\E[Z_n]$, we obtain a recurrence for $\E[X_n]$. Solving the recurrence with initial condition $\E[X_2] = X_2 = 2$, we obtain the result stated in the lemma.
	
	As $n \to \infty$, the digamma function $\Psi(n)$ in the first term is of order $\log{n}$. Meanwhile, the second fraction is of order $n^{3/2}$ according to the {\em Stirling's approximation}. Thus, we have $X_n/n^2$ converges to $36$ in $L_1$-space, which also suggests a weak law for $X_n$.
\end{proof}

In the second lemma, we derive the mixed moment of $Z_n$ and $Y_n$; namely, $\E[Z_n Y_n]$. Apparently, the variables $Z_n$ and $Y_n$ are not independent. Our strategy is to establish a recurrence on the expectation of the product of $Z_nY_n$.

\begin{lemma}
	\label{Lem:ZY}
	For $n \ge 2$, we have
	$$\E[Z_n Y_n] = \frac{64 \bigl(\log{n} + O(1)\bigr)\Gamma(n + 3/2) + 32/15}{\sqrt{\pi} \Gamma(n - 1)}.$$
	As $n \to \infty$, we have
	$$\frac{Z_n Y_n}{n^{5/2}\log{n}} \convL \frac{64}{\sqrt{\pi}}.$$
\end{lemma}

\begin{proof}
	By the definition of $Z_n$ ($Y_n$), we have the following almost-sure relations from $Z_{n - 1}$ ($Y_{n - 1}$) to $Z_n$ ($Y_n$), conditional on $\field_{n - 1}$ and $\indicator_{n, j}$:
	\begin{align}
	\label{Eq:recZ}
	Z_n &= Z_{n - 1} + (D_{n - 1, j} + 1)^2 - D^2_{n - 1, j} + 1 = Z_{n - 1} + 2 D_{n - 1, j} + 2.
	\\ Y_n &= Y_{n - 1} + (D_{n - 1, j} + 1)^3 - D^3_{n - 1, j} + 1 \nonumber 
	\\ \label{Eq:recY} &= Y_{n - 1} + 3 D^2_{n - 1, j} + 3 D_{n - 1, j} + 2.
	\end{align}
	Taking the product of Equations~(\ref{Eq:recZ}) and~(\ref{Eq:recY}), we get
	\begin{align*}
	Z_n Y_n &= Z_{n - 1}Y_{n - 1} + 3Z_{n - 1}D^2_{n - 1} + 3Z_{n - 1}D_{n - 1, j} + 2Z_{n - 1}
	\\ &\qquad{}+ 2Y_{n - 1}D_{n - 1, j} + 6D^3_{n - 1, j} +6D^2_{n - 1, j} + 4D_{n - 1, j}
	\\ &\qquad{}+ 2Y_{n - 1} +6D^2_{n - 1, j} + 6D_{n - 1, j} + 4.
	\end{align*}
	Averaging it out over $j$, we then have
	\begin{align*}
		\E[Z_n Y_n \given \field_{n - 1}] &= Z_{n - 1}Y_{n - 1} + 3Z_{n - 1}\sum_{j = 1}^{n - 1}\frac{D^3_{n - 1}}{2(n - 2)} + 3Z_{n - 1}\sum_{j = 1}^{n - 1}\frac{D^2_{n - 1}}{2(n - 2)} 
		\\ &\qquad{}+ 2Z_{n - 1} + 2Y_{n - 1}\sum_{j = 1}^{n - 1}\frac{D^2_{n - 1}}{2(n - 2)} + 6\sum_{j = 1}^{n - 1}\frac{D^4_{n - 1}}{2(n - 2)}
		\\ &\qquad{} + 6\sum_{j = 1}^{n - 1}\frac{D^3_{n - 1}}{2(n - 2)} + 4\sum_{j = 1}^{n - 1}\frac{D^2_{n - 1}}{2(n - 2)} + 2Y_{n - 1} 
		\\ &\qquad{}+6\sum_{j = 1}^{n - 1}\frac{D^3_{n - 1}}{2(n - 2)}  + 6\sum_{j = 1}^{n - 1}\frac{D^2_{n - 1}}{2(n - 2)} + 4
		\\ &= \frac{2n + 1}{2(n - 2)} Z_{n - 1}Y_{n - 1} + \frac{3}{2(n - 2)} Z^2_{n - 1} + \frac{3}{n - 2} X_{n - 1} 
		\\ &\qquad{}+ \frac{2(n + 1)}{n - 2} Y_{n - 1} + \frac{2n + 1}{n - 2} Z_{n - 1} + 4.
	\end{align*}
	The recurrence for $\E[Z_nY_n]$ is then obtained by taking another expectation with respect to $\field_{n - 1}$ and plugging in the results of $\E\left[Z^2_{n - 1}\right]$, $\E[X_{n - 1}]$, $\E[Y_{n - 1}]$, and $\E[Z_{n - 1}]$. Solving the recurrence with initial condition $\E[Z_2Y_2] = Z_2Y_2 = 2 \times 2 = 4$, we obtain the solution of $\E[Z_n Y_n]$. The $L_1$ convergence follows by applying the Stirling's approximation to $\E[Z_n Y_n]$.
\end{proof}

Note that the exact solution for $\E[Z_n Y_n]$ is derived. However, it can not be written in a closed form. Instead, it is the sum of four fraction terms involving gamma functions, digamma functions, and first order polygamma functions. We thus only present several leading terms of the solution in Lemma~\ref{Lem:ZY} for better readability. The complete solution is available upon request.

We are now ready to derive the third moment of $Z_n$. We use the results from Lemmata~\ref{Lem:X} and~\ref{Lem:ZY} as well as those from~\cite{ZhangArXiv}.

\begin{prop}
	For $n \ge 2$, we have
	$$\E\left[Z^3_n\right] = (n + 1)n(n - 1)\left(8 \log^3{n} + 24 \log^2{n} + O(\log{n})\right).$$
	As $n \to \infty$, we have
	$$\frac{Z^3_n}{\left(n \log{n}\right)^3} \convL 8.$$
\end{prop}

\begin{proof}
	Recall the almost-sure relation for $Z_n$. Raising Equation~(\ref{Eq:recZ}) to the third power on both sides, we have
	\begin{align*}
		Z_n^3 &= Z^3_{n - 1} + 8D^3_{n - 1, j} + 8 + 6Z^2_{n - 1}D_{n - 1, j} + 6Z^2_{n - 1} + 12D^2_{n - 1, j}Z_{n - 1} 
		\\ &\qquad{}+ 24D^2_{n - 1, j} + 12Z_{n - 1} + 24D_{n - 1} + 24 Z_{n - 1}D_{n - 1, j}.
	\end{align*}
	Average it out over $j$ to get
	\begin{align*}
	\E\left[Z_n^3 \given \field_{n - 1}\right] &= Z^3_{n - 1} + 8 \sum_{j = 1}^{n - 1}\frac{D^4_{n - 1}}{2(n - 2)} + 8 + 6Z^2_{n - 1}\sum_{j = 1}^{n - 1}\frac{D^2_{n - 1}}{2(n - 2)} + 6Z^2_{n - 1} 
	\\ &\qquad{}+ 12Z_{n - 1}\sum_{j = 1}^{n - 1}\frac{D^3_{n - 1}}{2(n - 2)} + 24\sum_{j = 1}^{n - 1}\frac{D^3_{n - 1}}{2(n - 2)} + 12Z_{n - 1}  
	\\ &\qquad{}+ 24\sum_{j = 1}^{n - 1}\frac{D^2_{n - 1}}{2(n - 2)} + 24Z_{n - 1}\sum_{j = 1}^{n - 1}\frac{D^2_{n - 1}}{2(n - 2)}
	&
	\\ &= \frac{n + 1}{n - 2} Z^3_{n - 1} + \frac{6n}{n - 2} Z^2_{n - 1} +  \frac{6}{n - 2}Z_{n - 1}Y_{n - 1} 
	\\ &\qquad{}+ \frac{12(n - 1)}{n - 2} Z_{n - 1} + \frac{4}{n - 2} X_n + \frac{12}{n - 2}Y_{n - 1} + 8.
	\end{align*}
	
Taking expectation on both sides, and plugging in all the results of lower moments, we obtain a recurrence on the third moment of $Z_n$. We solve the recurrence with the initial condition $\E\left[Z^3_2\right] = Z^3_2 = 8$, to get the stated result in the proposition, and the $L_1$ convergence of $Z_n^3$ after it is properly scaled immediately.
\end{proof}

Similar to $\E[Z_n Y_n]$, we get the exact solution for $\E\left[Z^3_n\right]$. However, the solution is even more complicated than that for $\E[Z_n Y_n]$, involving digamma functions, {\em Meijer G functions} and nested sums which can not be simplified to closed forms. However, the leading terms that we have developed are sufficient to characterize the asymptotic behavior of the skewness of $Z_n$.

\begin{theorem}
	\label{Thm:nonnormal}
	As $n \to \infty$, the distribution of $Z_n$ is skewed to the right. Hence, it is not normal.
\end{theorem}

\begin{proof}
	Recall the definition formula for ${\mathcal{S}}(Z_n)$:
	$${\mathcal{S}}(Z_n) = \frac{\E\left[Z_n^3\right] - 3 \E[Z_n] \E\left[Z_n^2\right] + 2 \bigl(\E[Z_n]\bigr)^3}{\bigl(\Var[Z_n]\bigr)^{3/2}}.$$
	Plugging in $\E\left[Z^3\right]$, $\E\left[Z^2\right]$ and $\E[Z_n]$, we find that the top three leading terms (of order $n^3 \log^3{n}$, $n^3 \log^2{n}$ and $n^3 \log{n}$) in the numerator are exactly canceled out, left with the highest nonzero term of order $n^3$. This is the same as the order of the leading term in the denominator. We thus come up with
	$$\mathcal{S}(Z_n) \sim \frac{-72 \gamma^2 + 164 \gamma - 41/3 + 4\pi^2 \gamma - 8 \gamma^3}{(16 - 2\pi^2/3)^{3/2}} \approx 2.7,$$
	as $n \to \infty$, and conclude that $Z_n$ does not converge to normal asymptotically. 
\end{proof}

\section{Zagreb index of caterpillars}

In mathematical chemistry, {\em caterpillar} is a popular model for representing the structure of benzoid hydrocarbon molecules~\cite{ElBasil, ElBasil1990}. In this section, we look into a class of random caterpillars by incorporating caterpillars with randomness. The class of random caterpillars considered grow in a preferential attachment manner, as described in Section~\ref{Sec:PORTs}. More precisely, at time $0$, there is a {\em spine} consisting of $m \ge 2$ (fixed) nodes, which were labeled with distinct numbers in $\{1, 2, \ldots, m\}$ from one end to the other. At each subsequent point, a {\em leaf} is linked to one of the spine
nodes with an edge, the probability being equal to its degree over the total degree of all spine nodes.

At time $n$, we denote the structure of a random caterpillar by $C_n$. We first give some graph invariants of $C_n$. The total number of nodes in $C_n$ is $(n + m)$, and total degree (of all nodes) is $(n + 2m - 2)$. Let $X_{n, j}$ be the number of leaves attached to spine node $i$, and let $D_{n, j}$ be the degree of spine node $i$. There is a instantaneous relation between $X_{n, j}$ and $D_{n, j}$. That is $D_{n, j} = X_{n, j} + 1$ for $j = 1,m$; $D_{n, j} = X_{n, j} +2$ for $j = 2, 3, \ldots, m - 1$. According to the evolution of $C_n$, the probability that the spine node $j$ is selected for recruiting a leaf at time $n$ is
$$\frac{D_{n - 1, j}}{\sum_{j = 1}^{m} D_{n - 1, j}}.$$
The Zagreb index of $C_n$ is given by
$$Z_n = {\tt Zagreb}(C_n) = \sum_{j = 1}^{m} D_{n, j} + n.$$

Note that we only account for caterpillars of $m \ge 2$ in this section, as the Zagreb index of a caterpillar of $m = 1$ ({\em star}) is deterministic; It is $(n^2 + n)$. In the next proposition, we derive the expectation of $Z_n$, and develop a weak law as well.

\begin{prop}
	For $n \ge 1$, the mean of the Zagreb index of caterpillars is
	$$\E[Z_n] = \frac{(3m - 4)n^2 + (4m - 3)(3m - 4)n + 2(2m - 3)}{(2m - 1)(m - 1)} + 2(2m - 3).$$
	As $n \to \infty$, we have 
	$$\frac{Z_n}{n^2} \convL \frac{3m - 4}{(2m - 1)(m - 1)}.$$
	This convergence takes place in probability as well.
\end{prop}

\begin{proof}
	We start with an almost-sure relation between $Z_n$ and $Z_{n - 1}$, conditional on $\field_{n - 1}$ and $\indicator_{n, j}$; that is,
	$$Z_n = Z_{n - 1} + (D_{n - 1, j} + 1)^2 - D^2_{n - 1, j} + 1 = Z_{n - 1} + 2D_{n - 1, j} + 2.$$
	This is identical to the almost-sure relation in Equation~(\ref{Eq:recZ}). Taking the expectation with respect to $\indicator_{n, j}$, we get
	\begingroup
	\allowdisplaybreaks
	\begin{align*}
	\E[Z_n \given \field_{n - 1}] &= Z_{n - 1} + 2 \sum_{j = 1}^{m} D_{n - 1, j} \frac{D_{n - 1, j}}{\sum_{j = 1}^{m} D_{n - 1, j}} + 2
	\\ &= Z_{n - 1} + \frac{2}{n + 2m - 3} \sum_{j = 1}^{m} D^2_{n - 1, j} + 2
	\\ &= Z_{n - 1} + \frac{2}{n + 2m - 3} \bigl(Z_{n - 1} - (n - 1)\bigr) + 2
	\\ &= \frac{n + 2m - 1}{n + 2m - 3} \, Z_{n - 1} + \frac{4(m - 1)}{n + 2m -3}.
	\end{align*}
	\endgroup
	Taking another expectation with respect to $\field_{n - 1}$, we obtain a recurrence for $\E[Z_n]$:
	$$\E[Z_n] = \frac{n + 2m - 1}{n + 2m - 3} \, \E[Z_{n - 1}] + \frac{4(m - 1)}{n + 2m -3}.$$	
	We solve the recurrence with initial condition $\E[Z_0] = Z_0 = 4m - 6$, to get the result stated in the proposition. Both $L_1$ convergence and in-probability convergence of $Z_n / n^2$ are obtained immediately.
\end{proof}

Towards the second moment of $Z_n$, we need the mean of a new topological index---the total of cubic degrees of all nodes in $C_n$. Let us denote this index by $Y_n$, and we have
$$Y_n := \sum_{j = 1}^{m} D^3_{n, j} + n.$$
The mean of $Y_n$ is given in the next lemma.

\begin{lemma}
	For $n \ge 1$, the mean of $Y_n$ is given by
	\begin{align*}
	\E[Y_n] &= \frac{3(2m - 3)n^3 + (27m^2 - 60m + 27)n^2}{(2m - 1)m(m - 1)}
	\\ &\qquad{} + \frac{(40m^3 - 111m^2 + 86m - 18)n}{(2m - 1)m(m - 1)} + 2(4m - 7).
	\end{align*}
\end{lemma}

\begin{proof}
	We consider an almost-sure relation between $Y_n$ and $Y_{n - 1}$ analogous to Equation~(\ref{Eq:recY}):
	$$Y_n = Y_{n - 1} + 3 D^2_{n - 1, j} + 3 D_{n - 1, j} + 2.$$
	Taking the expectation with respect to $\indicator_{n, j}$, we get
	\begin{align*}
	\E[Y_n \given \field_{n - 1}] &= Y_{n - 1} + 3 \sum_{j = 1}^{m} D^2_{n - 1, j} \frac{D_{n - 1, j}}{\sum_{j = 1}^{m} D_{n - 1, j}} + 3 \sum_{j = 1}^{m} D_{n - 1, j} \frac{D_{n - 1, j}}{\sum_{j = 1}^{m} D_{n - 1, j}} 
	\\ &\qquad{}+ 2
	\\ &= Y_{n - 1} + \frac{3}{n + 2m - 3} \sum_{j = 1}^{m} D^3_{n - 1, j} + \frac{3}{n + 2m - 3} \sum_{j = 1}^{m} D^2_{n - 1, j} + 2
	\\ &= Y_{n - 1} + \frac{3\bigl(Y_{n - 1} - (n - 1)\bigr)}{n + 2m - 3} + \frac{3\bigl(Z_{n - 1} - (n - 1)\bigr)}{n + 2m - 3}  + 2.
	\end{align*}
	The recurrence for $\E[Y_n]$ is obtained by taking another expectation with respect to $\field_{n - 1}$, and by plugging the result of $\E[Z_{n - 1}]$. We solve the recurrence for $\E[Y_n]$ with initial condition $\E[Y_0] = Y_0 = 8m - 14$, and obtain the stated result.
\end{proof}

We are now ready to calculate the second moment of $Z_n$ in the next proposition.

\begin{prop}
	For $n \ge 1$, the second moment of the Zagreb index of caterpillars is given by
	$$
	\E\left[Z^2_n\right] = \frac{(9m^2 - 3m - 16)n^4 + 2(3m - 2)(12m^2 - 7m - 17)n^3 + O\left(n^2\right)}{(2m + 1)(2m - 1)m(m - 1)}.
	$$
\end{prop}

\begin{proof}
	We begin with squaring the almost-sure relation (between $Z_n$ and $Z_{n - 1}$).
	\begin{align*}
	Z^2_n &= (Z_{n - 1} + 2D_{n - 1, j} + 2)^2
	\\ &= Z^2_{n - 1} + 4D^2_{n - 1, j} + 4 + 4Z_{n - 1}D_{n - 1, j} + 4Z_{n - 1} + 8D_{n - 1, j}
	\end{align*}
	We take the average over $j$ to get
	\begin{align*}
	\E\left[Z^2_n \given \field_{n - 1}\right] &= Z^2_{n - 1} + 4 \sum_{j = 1}^{m} D^2_{n - 1, j} \frac{D_{n - 1, j}}{\sum_{j = 1}^{m} D_{n - 1, j}} + 4 
	\\&\qquad{}+ 4Z_{n - 1}\sum_{j = 1}^{m} D_{n - 1, j} \frac{D_{n - 1, j}}{\sum_{j = 1}^{m} D_{n - 1, j}} + 4Z_{n - 1} 
	\\&\qquad\qquad{}+ 8\sum_{j = 1}^{m} D_{n - 1, j} \frac{D_{n - 1, j}}{\sum_{j = 1}^{m} D_{n - 1, j}}
	\\ &= Z^2_{n - 1} + \frac{4\bigl(Y_{n - 1} - (n - 1)\bigr)}{n + 2m - 3} + 4 + \frac{4Z_{n - 1}\bigl(Z_{n - 1} - (n - 1)\bigr)}{n + 2m - 3}
	\\ &\qquad{} + 4 Z_{n - 1} +\frac{8\bigl(Z_{n - 1} - (n - 1)\bigr)}{n + 2m - 3}.
	\end{align*}
	We take another expectation with respect to $\field_{n - 1}$ to get the recurrence for $\E\left[Z^2_{n}\right]$. Solving the recurrence with the initial condition $Z^2_0 = \E\left[Z^2_0\right] = (4m - 6)^2$, we obtain the result stated in the proposition.
\end{proof}

The variance of $Z_n$ is the obtained immediately by taking the difference between $\E\left[Z^2_n\right]$ and $\E^2[Z_n]$.

\begin{corollary}
	For $n \ge 1$, the variance of the Zagreb index of caterpillars is
	$$\Var[Z_n] = \frac{(6m^3 - 22m^2 + 29m - 16)n^4 + O(n^3)}{(2m - 1)^2(m - 1)^2(2m + 1)m}.$$
\end{corollary}

\section{Conclusion}

In this article, we investigate the Zagreb index of three random networks, a class of random graphs extended from RRTs, PORTs, and preferential attachment caterpillars. For the first class, we show that the Zagreb index scaled by $\sqrt{n}$ is asymptomatically normal. For the second class, we prove that the asymptotic distribution of the Zagreb index is not normal by showing that the distribution is skewed to the right. For the third class, we find that the second moment and the variance of the Zagreb index have the same order. We thus conjecture that its asymptotic distribution is not normal as well.

One of the possible future work is to study the Zagreb index of more general preferential attachment networks~\cite{Barabasi}, a class of networks extended from PORTs. The recurrence method seems not amenable owing to the non-uniformity of the sampling distribution. One alternative approach is to exploit the degree profile of preferential attachment networks developed in~\cite{Pekoz}. We will conduct the investigation in this direction and report the results elsewhere.



\end{document}